\documentclass[twoside]{amsart}
\usepackage{amssymb,verbatim}
\usepackage{amsmath}
\usepackage{amsthm}
\usepackage{graphicx}
\usepackage{color}
\usepackage[all]{xy}

\newtheorem{theorem}{Theorem}[section]
\newtheorem{proposition}[theorem]{Proposition}

\newtheorem{Proposition}[theorem]{Proposition}
\newtheorem{Lemma}[theorem]{Lemma}

\theoremstyle{definition}
\newtheorem{Definition}[theorem]{Definition}
\newtheorem{example}[theorem]{Example}
\theoremstyle{remark}
\newtheorem{remark}[theorem]{Remark}

\newtheorem{assumption}[theorem]{Assumption}

\newcommand{\cO}{{\mathcal O}}

\newcommand{\N}{\mathbb N}
\newcommand{\Q}{\mathbb Q}
\newcommand{\Z}{\mathbb Z}
\newcommand{\C}{\mathbb C}
\newcommand{\R}{\mathbb R}
\newcommand{\proj}{\mathbb P}

\newcommand{\bb}[1]{\mathbb{#1}}
\newcommand{\m}[1]{\mathcal{#1}}

\newcommand{\f}{\varphi}
\newcommand{\ra}{\rightarrow}
\newcommand{\lra}{\longrightarrow}

\newcommand{\lin}{\sim}

\DeclareMathOperator{\Spec}{Spec}

\DeclareMathOperator{\Proj}{Proj}

\begin{document}


\title[Lifting WBU]{Lifting Weighted Blow-ups} 
\author[Andreatta]{Marco Andreatta}

\thanks{
I like to thank Alessio Corti for suggesting to use Schlessinger theorem in the proof of Proposition \ref{cyclic} and for providing Example \ref{3cyclic}. I thank Roberto Pignatelli and Luca Tasin for very helpful conversations. I was supported by the MIUR grant PRIN-2010.}

\address{Dipartimento di Matematica, Universit\`a di Trento, I-38123
Povo (TN)} 
\email{marco.andreatta@unitn.it}


\begin{abstract} 
Let $f: X \ra Z$ be a local, projective, divisorial contraction between normal varieties of dimension $n$ with $\Q$-factorial singularities.

Let $Y \subset X$ be a $f$-ample Cartier divisor and assume that $f_{|Y}:  Y \ra W$ has a structure of a weighted blow-up. We prove that $f: X \ra Z$, as well, has a structure of weighted blow-up.

As an application we consider a local projective contraction $f: X \ra Z$ from a variety $X$ with terminal $\Q$-factorial singularities, which contracts a prime divisor $E$ to an isolated $\Q$-factorial singularity $P\in Z$, such that  $-(K_X + (n-3)L)$ is $f$-ample, for a $f$-ample Cartier divisor $L$ on $X$. 
We prove that $(Z,P)$ is a hyperquotient singularity and $f$ is a weighted blow-up. 
\end{abstract}

\maketitle

\section{Introduction}

Let $X$ be a normal variety over $\C$ and $n = \dim X$. A  \emph{contraction} is a surjective morphism  $\f: X \ra Z$ with connected fibres onto a normal variety $S$.  
If $Z$ is affine then $f:X\ra Z$ will be called a \emph{local contraction}.

We always assume that $f$ is \emph{projective}, that is we assume the existence of $f$-ample Cartier divisors $L$. 

If $f$ is birational and its exceptional set is an irreducible divisor then it is called \emph{divisorial}. 
We say that the contraction is \emph{$\Q$-factorial} if $X$ and $Z$ have $\Q$-factorial singularities. Note that if $X$ is $\Q$-factorial and $f$ is a divisorial contraction of an extremal ray (in the sense of Mori Theory) then $Z$ is also $\Q$-factorial (see Corollary 3.18 in \cite{KollarMori}).

\medskip
A fundamental example of local contraction in Algebraic Geometry is the blow-up of  $\C^n = Spec\ \C[x_1, ...,x_n]$ at $0$. More generally, given $\sigma=(a_1,\ldots,a_n) \in \mathbb N^n$ such that $a_i >0$ and $m \in \N$, one can define the \textit{$\sigma$-blow-up} (or the weighted blow-up with weight $\sigma$) of a  hyperquotient singularity $Z : ((g=0)\subset  \C^n )/ \Z_m(a_1, ..., a_n)$. The definition is given in Section \ref{s_weighted}, in accordance with Section 10 in \cite{KollarMori92}.

\medskip
The main goal of the paper is to prove the following Theorem.
\begin{theorem}
\label{lifting}
Let $f: X \ra Z$ be a local, projective, divisorial and $\Q$-factorial contraction, which contracts an irreducible divisor $E$ to an isolated $\Q$-factorial singularity $P\in Z$. Assume that $dim X \geq 4$.

Let $Y \subset X$ be a $f$-ample Cartier divisor such that $f' = f_{|Y}: Y \ra f(Y) = W$  is a $\sigma'=  (a_1, \ldots,a_{n-1})$-blow-up, $\pi_{\sigma'}:  Y \ra W$.

Then $f: X \ra Z$ is a $\sigma=  (a_1, \ldots,a_{n-1}, a_n)$-blow-up, 
$\pi_{\sigma}:  X \ra Z$, where $a_n$ is such that $Y \sim_f -a_n E$ ($\sim_f$ means linearly equivalent over $f$).

\end{theorem}

\bigskip
We apply the above Theorem to the study of birational contractions which appear in a  Minimal Model Program (MMP) with scaling on polarized pairs. 

\smallskip
More precisely, if $X$ is a variety with terminal $\Q$-factorial singularities and  $L$ is an ample Cartier 
divisor on $X$, the pair $(X,L)$ is called a {\sl Polarized Pair}.
Given a non negative rational number $r$,  there exists an effective $\Q$-divisor $\Delta^r$ on $X$ such that
$\Delta^r \sim_{\Q} r L$ and $(X, \Delta^r)$ is  Kawamata log terminal.
Consider the pair $(X, \Delta^r)$ and the $\Q$-Cartier divisor $K_X + \Delta^r  \sim_{\Q} K_X + r L$.

By Theorem 1.2 and Corollary 1.3.3 of \cite{BCHM}  we can run a {\sl $K_X + \Delta^r$-Minimal Model Program (MMP) with scaling}.
This type of MMP was studied in deeper details in the case $r \geq (n-2)$ in \cite{And13}.

\smallskip
To perform such a program one needs to understand local birational maps (divisorial or small contractions), $f: X \ra Z$, which are contractions of an extremal rays $R := \R^+ [C] \subset N_1(X/Z)$, where $C$ is a rational curve such that $(K_X + rL)^. C< 0$ for a $f$-ample Cartier divisor $L$. 
We will call these maps Fano-Mori contractions or {\sl contractions for a MMP}.

In \cite{AnTa} we classify local birational contractions for a MMP if $r \geq (n-2)$: they are $\sigma$-blow-up of a smooth point with $\sigma = (1,1,b,..., b)$, where $b$ is a positive integer.

In \cite{AnTa15}, Theorem 1.1,  we prove that if $r > (n-3) >0$ then one can find a general divisor $X' \in |L|$ which is a variety with at most $\Q$-factorial terminal singularities and such that $f_{|X'} : X' \to f(X')=:Z'$ is a contraction of an extremal ray $R' := \R^+ [C']$ such that $(K_{X'}+ (r - 1) L')^. C'< 0$, where $L' := L_{|X'}$.

On the other hand a very hard program, aimed to classify local divisorial contractions to a point for a MMP in dimension $3$, has been started long ago by Y. Kawamata (\cite{Ka}); it was further carried on by M. Kawakita, T. Hayakawa and J. A. Chen (see, among other papers, \cite{Kaw01}, \cite{Kaw02}, \cite{Kaw03}, \cite{Kaw05}, \cite{Kaw12}, \cite{Hayakawa99}, \cite{Hayakawa00}, \cite{Hayakawa05}, \cite{Chen}). They are all weighted blow-ups of (particular) cyclic quotient or hyperquotient singularities and this should be the case for the few remaining ones.  It is reasonable to make the following:

\begin{assumption} \label{ass}
The divisorial contractions to a point for a MMP in dimension $3$ are weighted blow-up. 
\end{assumption}

The next result is a consequence, via a standard induction procedure, called {\it Apollonius method}, of Theorem \ref{lifting}, the above quoted Theorem 1.1 in \cite{AnTa15} and Assumption (\ref{ass}) in dimension $3$.

\begin{theorem} 
\label{MMP}
Let $X$ be a variety with $\Q$-factorial terminal singularities of dimension $n \geq 3$ and 
let $f : X \ra Z$ be a local, projective, divisorial contraction which contracts a prime divisor $E$ to an isolated $\Q$-factorial singularity $P\in Z$
such that  $-(K_X + (n-3)L)$ is $f$-ample, for a $f$-ample Cartier divisor $L$ on $X$. 

Then $P\in Z$ is a hyperquotient singularity.

Moreover, if we assume that  \ref{ass} holds, 
$f$ is a weighted blow-up.

\end{theorem}

\section{Weighted blow-ups}
\label{s_weighted}

We recall the definition of weighted blow-up, our notation is compatible with 
that of Section 10 in \cite{KollarMori92} and of Section 3 in \cite{Hayakawa99}. 

Let  $\sigma=(a_1,\ldots,a_n) \in \mathbb N^n$ such that $a_i >0$ and $\gcd(a_1,\ldots,a_n)=1$; let $M=\mathrm{lcm}(a_1,\ldots,a_n)$.

\medskip
{\sl The weighted projective space} with weight $(a_1,\ldots,a_n)$, denoted by  $ \bb P(a_1, \ldots,a_n)$,
can be defined either as:

$$\bb P(a_1, \ldots,a_n): = (\C^n - \{0\}) /\C^*,$$
where $\xi \in \C^*$ acts by $\xi (x_1, ..., x_n) = (\xi^{a_1}x_1, ..., \xi^{a_n}x_n)$.

Or as:
$$ \bb P(a_1, \ldots,a_n) := Proj_{\C}\C[x_1, ..., x_n],$$
where $\C[x_1, ..., x_n]$ is the polynomial algebra over $\C$ graded by the condition $deg(x_i) = a_i$, for $i =1,..., n$.

\medskip
A {\sl cyclic quotient singularity}, denoted by $\C^n / \Z_m(a_1, ..., a_n):= X$, is an affine variety definite as the quotient of $\C^n$ by the action $\epsilon: (x_1, ..., x_n) \ra (\epsilon^{a_1}x_1, ..., \epsilon^{a_n}x_n)$, where $\epsilon$ is a primitive $m$-th root of unity. Equivalently $X$ is isomorphic to the spectrum of the ring of invariant monomials under the group action, $Spec\ \C[x_1, ...,x_n]^{\Z_m}$.

\medskip
Let $Q \in Y: (g=0)  \subset \C^{n+1}$  be a hypersurface singularity with a $\Z^m$ action. The point $P \in Y/\Z^m := X$ is called a {\sl hyperquotient singularity}. In suitable local analytic coordinates the action on $Y$ extends to an action on $\C^{n+1}$ (in fact it acts on the tangent space $T_{Y,Q}$) and we can assume that $\Z_m$ acts diagonally by $\epsilon: (x_0, ..., x_n) \ra (\epsilon^{a_0}x_0, ..., \epsilon^{a_n}x_n)$, where  $\epsilon$ is a primitive $m$-th root of unity.
Since $Y$ is fixed by the action of $\Z_m$, it follows that $g$ is an eigenfunction, so that  $\epsilon: g \ra \epsilon ^eg$. We define the {\sl type} of the hyperquotient singularity $P \in X$ with the symbol ${1\over m} (a_0,..., a_n;e)$. Note that if $m=1$ this is simply a hypersurface singularity, while if $g = x_0$ this is a cyclic quotient singularity.

\bigskip
Let  $X = \C^n / \Z_m(a_1, ..., a_n)$ be a cyclic quotient singularity and consider the rational map
$$
\f: X \to \bb P(a_1, \ldots,a_n)
$$
given by $(x_1,\ldots,x_n) \mapsto (x_1:\ldots:x_n)$.

\begin{Definition}
\label{weighted}
The \textit{weighted blow-up} of $X= \C^n / \Z_m(a_1, ..., a_n)$ with weight $\sigma = (a_1, ..., a_n)$ (or simply the $\sigma$-blow-up), $\overline X$,  is defined as the closure in 
$X \times \bb P(a_1, \ldots,a_k)$ of the graph of $\f$, together with the morphism $\pi_{\sigma}: \overline X \to X$ given by the projection on the first factor.
\end{Definition}

\medskip
The weighted blow-up can be described by the theory of torus embeddings, as in section 10 of \cite{KollarMori92}. 
Namely, let $e_i = (0, ...,1, ...,0)$ for $i = 1,..., n$ and $e= 1/m(a_1, ..., a_n)$. Then  $X$ is the toric variety which corresponds to the lattice 
$\Z e_1 + ... +\Z e_n + \Z e$ and the cone $C(X) = \Q_+ e_1 + ... +\Q_+ e_n$ in $\Q^n$, where $\Q_+ = \{ z \in \Q : z \geq 0\}$.

$\pi_{\sigma}: \overline X \to X$ is the proper birational morphism from the normal toric variety $\overline X$ corresponding to the cone decomposition of $C(X)$ consisting of $C_i = \Sigma_{j \not=   i} \Q _+ e_j + \Q _+ e$, for $i = 1,..., n$, and their intersections.

\bigskip
The following facts can be easily checked in many ways, for instance via toric geometry (see also section 10 in \cite{KollarMori92} or section 3 in \cite{Hayakawa99}). 

\begin{itemize}

\smallskip
\item The map $\pi_{\sigma}$ is birational and contracts an exceptional irreducible divisor $E \cong \mathbb P(a_1,\ldots,a_k)$ to $0 \in X$. 

\item
Let $(y_1: \ldots :y_n)$ be homogeneous coordinates on $\bb P(a_1, \ldots,a_n)$. For any $1\le i \le k$ consider the open affine subset $U_i=\overline X \cap \{y_i \ne 0  \}$; these affine open subset are described as follows:

$$U_i \cong  \Spec \C[\bar x_1, \ldots, \bar x_n]/\bb Z_{a_i}(-a_1,\ldots,m,\ldots,-a_n)$$

The morphism ${\f_{\sigma}}_{|U_i} :U_i \to X$ is given by

$$
(\bar x_1, \ldots, \bar x_n) \mapsto (\bar x_1 \bar x_i^{a_1/m},\ldots, \bar x_i^{a_i/m}, \ldots, \bar x_k \bar x_i^{a_k/m}).$$

\item
 In the affine set $U_i$ the divisor $E$ is defined by $\{\bar x_i=0  \}$; it is a $\Q$-Cartier divisor and $\cO_{\overline X}(-aE)\otimes \cO_E= \cO_{\proj} (ma)$, for $a$ divisible by $\Pi a_i$. \\
$H:= -ME$ is actually Cartier, it is generated over $\pi_{\sigma}$ by global sections and it is the generator of $Pic(\overline X /X)  = \Z = <H>$.

\item
Let $L = a H$, for $a$ a positive integer; clearly $L$ is $\sigma$-ample. We have 
$$ R^1{\pi_{\sigma}}  _*\cO_Y(iL) = H^1(\overline X, iL) =  0 $$ 
for every $i\in \Z$.

\end{itemize}

\bigskip
We now use Grothendieck's language to give a different characterization of the  $\sigma$-weighted blow-up.

For $a$ a positive integer let  $L = a H = -aM E$. $L$ is a $\pi_{\sigma}$-ample Cartier divisor. 

Consider the graduated $\C[x_1, ...,x_n]^{\Z_m}$-algebra  $\bigoplus_{d\ge 0} \pi_*\m O_X(dL)$. 
The construction in section (8.8) of \cite{EGA II}, gives
$$
\overline{X}= \Proj_X \big( \cO_X \oplus \bigoplus_{d >  0} \pi_*\m O_X(dL)\big) \to X.
$$

\smallskip
Consider now the function 
$$
\sigma\textrm{-wt}: \bb C[x_1,\ldots,x_n] \to \bb \Q
$$
defined as follows. On a monomial $M=x_1^{s_1}\ldots x_n^{s_n}$ we put $\sigma\textrm{-wt}(M):=\sum_{i=1}^n s_i a_i/m$.
For a general $f= \sum_I \alpha_I M_I$, where $\alpha_I \in \bb C$ and $M_I$ are monomials, we set 
$$
\sigma\textrm{-wt}(f):=\min \{\sigma\textrm{-wt}(M_I) : \alpha_I \ne 0 \}.
$$

\begin{Definition}
\label{weightedId}
For a rational number $k$ the $\sigma$-weighted ideal $I^{\sigma}(k)$ is defined as:
$$ 
I^{\sigma}(k) = \{g \in \C[x_1,\ldots,x_n] : \sigma\textrm{-wt}(g)\ge k  \}= (x_1^{s_1}\cdots x_n^{s_n} : \sum_{j=1}^n s_ja_j /m\ge k).
$$

$I^{\sigma}(k)$ is a an ideal in $\C[x_1,\ldots,x_n] $ and therefore also in $\C[x_1,\ldots,x_n] ^{\Z_m}$; in particular $\C[x_1, ...,x_n]^{\Z_m} \oplus \bigoplus_{k\in \N, d > 0} I^{\sigma}(k)$
is a  $\C[x_1, ...,x_n]^{\Z_m}$-graded module.
\end{Definition}

The next Lemma follows straightforward from the above discussion; see also Lemma 3.5 in \cite{Hayakawa99}.

\begin{Lemma}
 \label{push-forward}
Let $\pi_{\sigma}: \overline X \to X$ be a $\sigma$-blow-up, $E$ the exceptional divisor; let $D$ be the $\Q$-Cartier Weil divisor defined by a $\Z_m$-semi invariant $f \in \C[x_1, ..., x_n]$. Then we have 
$$\pi_{\sigma}^*(D) = \overline D + (\sigma\textrm{-wt}(f)) E,$$

where $\overline D$ is the proper transform of $D$.

In particular, for every integer $a$, we have  $\pi_*\m O_{\overline X}(-aE) = I^{\sigma}(a)$.
\end{Lemma}

\medskip
The Grothendieck set-up and the Lemma imply immediately the following characterization of weighted blow-up.

\begin{Proposition} 
\label{ideal}
Let $X= \C^n / \Z_m(a_1, ..., a_n)$ and $b$ a positive integer multiple of $M = \mathrm{lcm}(a_1,\ldots,a_n)$. 
The \textit{weighted blow-up} of $X$ with weight $\sigma$ defined above, $\pi_{\sigma}: \overline X \to X$,
is given by
$$
\overline{X}= \Proj_X \big( \cO_X \oplus \bigoplus_{d\in \N, d\ > 0} I^{\sigma}(db)\big).
$$
\end{Proposition}

\begin{remark} The above characterization of $\overline X$ does not depend on the the choice of $b$ as a positive multiple of $M$;
in fact taking $\Proj$ of {\sl truncated} graded algebras we obtain isomorphic objects (see for instance Exercise 5.13 or 7.11, Chapter II in \cite{Hartshorne}).

Note that it is not true that $I^{\sigma}(db) = I^{\sigma}(b)^d$: see for instance Example 3.5 in \cite{AnTa}.
However this is true if $b$ is chosen big enough; this can be proved, for instance, following the proof of Theorem 7.17 in \cite{Hartshorne}.

If this is the case we have that $\overline{X}= \Proj_X \big( \cO_X \oplus \bigoplus_{d\in \N, d\ > 0} I^{\sigma}(b)^d\big)$; that is $\overline X$ is the blowing-up of  $X= \C^n / \Z_m(a_1, ..., a_n)$ with respect to the coherent ideal $I^{\sigma}(b)$ (see the definition in Section 7, Chapter II, \cite{Hartshorne}).

\end{remark}

\medskip
\begin{Definition} Let $X :  ( (g=0)  \subset \C^{n+1})/ \Z_m(a_0, ...., a_n)$ be a hyperquotient singularity and let $\pi : \overline{\C^{n+1}/ \Z_m(a_0, ...., a_n)}  \ra \C^{n+1}/ \Z_m(a_0, ...., a_n)$ be the $\sigma = (a_0, ..., a_n)$-blow-up. Let $\overline X$ be the proper transform of $X$ via $\pi$ and call again, by abuse, $\pi$ its restriction to $\overline X$.  
Then $\pi: \overline X \ra X$ is also called the  {\sl  weighted blow-up of $X$ with weight $\sigma = (a_1, ..., a_n)$} (or simply the $\sigma$-blow-up).
\end{Definition} 

The above Proposition \ref{ideal}, together with Corollary 7.15,  Chapter II, \cite{Hartshorne}, implies the following.
\begin{Proposition} 
\label{ideal2}
Let $X :  ( (g=0)  \subset \C^{n+1})/ \Z_m(a_0, ...., a_n)$ be a hyperquotient singularity and let $i: X \ra \C^{n+1}/ \Z_m(a_0, ...., a_n)$ be the inclusion. 

Then 
$$\overline X =  \Proj_X \big( \cO_X \oplus\bigoplus_{d\in \N, d > 0} J^{\sigma}(db)\big) \ra X,$$
where $J^{\sigma}(db) := i^{-1}\big( I^{\sigma}(db) \big)^. \cO_X $.

If $b$ is big enough then 
$$\overline X =  \Proj_X \big( \cO_X \oplus\bigoplus_{d\in \N, d > 0} J^{\sigma}(b)^d\big) \ra X.$$

\end{Proposition}

\section{Lifting cyclic quotient singularities}

In this section we consider affine varieties $Z$ and $W$; we think at them as germs of complex spaces around a point $P$, $(Z,P)$ and $(W,P)$.
We assume that $P \in Z$ is an isolated $\Q$-factorial singularities; $\Q$-factoriality in this case depends on the analytic type of the singularity.

\begin{proposition}
\label{cyclic}
Let $Z$ be an affine variety of dimension $n\geq 4$ and assume that $Z$ has an isolated $\Q$-factorial singularity at $P \in Z$. 

Assume that $(W,P) \subset (Z,P) $ is a Weil divisor which is a cyclic quotient singularity, i.e.  $W = \C^{n-1} / \Z_m(a_1, ..., a_{n-1})$. \\
Then $Z$ is a cyclic quotient singularity, i.e.  $Z =\C^n / \Z_m(a_1, ..., a_{n-1}, a_n)$, where $a_n\in \Z$ is defined in the proof.
\end{proposition}

\begin{proof} Assume first that $W$ is a Cartier divisor, i.e. $W$ is given as a zero locus of a regular function $f$, 
$W :(f = 0)\subset Z$. 
The map $f:Z \ra \C$ is flat, since dim$_{\C}\C=1$.  Quotient singularities of dimension bigger or equal then three are rigid, by a fundamental theorem of M. Schlessinger (\cite{Schl}). Since $Z$ has an isolated singularity and $dimW = n-1 \geq 3$, it implies that  $W$ is smooth, i.e. $m=1$. A variety containing a smooth Cartier divisor is smooth along it, therefore, eventually shrinking around $P$, $Z$ is also smooth.

\smallskip
In the general case, since $Z$ is  $\Q$-factorial, we can assume that there exists a minimal positive integer $r$ such that  $rW$ is Cartier ($r$ is the index of $W$).  Following Proposition 3.6 in \cite{Re87}, we can take a Galois cover $\pi:Z' \ra Z$, with group $\Z_r$, such that $Z'$ is normal, $\pi$ is etale over $Z \setminus P$, $\pi ^{-1}(P)=: Q$ is a single point and the $\Q$-divisor $\pi^*W := W'$ is Cartier, $W': (f'=0) \subset Z'$.

Our assumption on $W$ implies that $r | m$, i.e. $m = r ^. s$, and $W' = \C^{n-1} / \Z_s(a_1, ..., a_{n-1})$. By the first part of the proof we have that $s=1$, i.e.  $W'$ and $Z'$ are smooth. 

Taking possibly a smaller neighborhood of $Q$, we can assume that, if $W' = \C^{n-1}$ with coordinates $(x_1, ..., x_{n-1})$, then $Z' = \C^n$,  with coordinates $(x_1, ..., x_{n-1}, x_n)$, where $x_n:= f'$.

The action of $\Z_m$ on $\C^n$, which extends the one on $\C^{n-1}$, fixes $W'$, therefore $f'$ is an eigenfunction;  that is for a primitive $m$-root of unity $\epsilon$ there exists $a_n \in \N$ such that $\epsilon: f' \ra \epsilon^{a_n} f'$.

Therefore the Galois cover $\pi:Z' = \C^n \ra Z$ is exactly the cover of the cyclic quotient singularity $Z=\C^n / \Z_m(a_1, ..., a_{n-1}, a_n)$.

\end{proof}

\begin{remark} 
If $n=3$ the above Proposition is false, as the following example shows.
\end{remark}

\begin{example}
\label{3cyclic}
Let $Z' = \C^4 / \Z_r(a, -a, 1, 0)$; let $(x,y,z,t)$ be coordinates in $\C^4$ and assume $(a,r)=1$. Let $Z \subset Z'$ be the hypersurface given as the zero set of the function $f:= xy + z^{rm} + t^n$, with $m \geq 1$ and $n \geq 2$. This is a terminal singularity which is not a cyclic quotient (it is a terminal hyperquotient singularity); in the classification of terminal singularities it is described in Theorem (12.1) of \cite{Mo82} (see also section 6 of \cite{Re87}).

However the surface $W := Z \cap (t=0)$, which is the surface in $\C^3 / \Z_r(a, -a, 1)$ given as the zero set of $(xy + z^{rm})$, is a cyclic quotient singularity of the type $\C^2 / \Z_{r^2 m}(a,rm-a)$. 

We give a proof of this last fact for the interested reader.
Let $\overline W$ be the surface in $\C^3$, with coordinate $(x,y,z)$, given as the zero set of the function $xy + z^{rm}$. $\overline W$ has a singularity of type $A_{rm-1}$, which is a cyclic quotient singularity of type $\overline W = \C^2 / \Z_{rm}(1, -1)$.

Let $(\xi, \eta)$ be the coordinate of $\C^2$ and let $\epsilon = e ^{2\pi i \over r^2m}$ a $r^2m$ root of unit; note that $\epsilon ^r$ is a $rm$ root of unit. 
The action of $ \Z_{rm}$ on $\C^2$ can be described  as $\epsilon^r  (\xi, \eta) = (\epsilon^{r}\xi, \epsilon^{-r} \eta)$.
A base for 
$\C[\xi, \eta]^{\Z_{rm}}$, the spectrum of the ring of invariant monomials under the group action, is given by 
$(\xi^{rm}, \eta^{rm},\xi^. \eta)$ and therefore
$\overline W = Spec(\xi^{rm}, \eta^{rm},\xi^. \eta)$. Let $(x,y,z) = (\xi^{rm}, \eta^{rm},\xi^. \eta)$, then $W$ is obtained as the quotient of $\overline W$ by the action of $\Z_r$ with weights $(a, -a, 1)$ given by  
$\epsilon^{rm} (x, y, z) = (\epsilon^{rma}x , \epsilon^{-rma}y, \epsilon^{rm}z)$. 
It is easy to check that this action can be lifted directly to $\C^2$ as the action: $\epsilon  (\xi, \eta) = (\epsilon^{a}\xi, \epsilon^{rm - a} \eta)$. This extends the previously defined $\Z_{rm}$-action on $\C^2$ and has $W$ as quotient.
\end{example}

\begin{proposition}
\label{hyperquotient}
Let $Z$ be an affine variety of dimension $n\geq 4$ with an isolated $\Q$-factorial singularity at $P \in Z$. Assume also that $(W,P) \subset (Z,P) $  is a Weil divisor which has a hyperquotient singularity at $P$. \\
Then $(Z,P)$ is a hyperquotient singularity.
\end{proposition}

\begin{proof} Let $W: (g=0) \subset  \C^{n} / \Z_m(a_1, ..., a_n)$.

As in the previous proof we assume first that $W$ is a Cartier divisor, i.e. $W$ is given as the zero locus of 
a regular function $f$. The map $f: Z \ra \C$ is flat and it gives a deformation of $W$. Since $W$ is a hypersurface singularity, its infinitesimal deformations are all embedded deformations, i.e. they extend to a deformation of the ambient space. That is, there exists a flat map $\tilde f: \tilde Z \ra \C$, such that $\tilde f ^{-1} (0) = \C^{n} / \Z_m(a_1, ..., a_n) $, $Z$ is a hypersurface in $\tilde Z$, i.e.  $Z : (\tilde g = 0) \subset \tilde Z$, and $\tilde f_{|Z} = f$. 

By Schlessinger's theorem (\cite{Schl}) this deformation $\tilde f$ is rigid, therefore $\tilde Z= \C^{n} / \Z_m(a_1, ..., a_n)\times \C = \C^{n+1} / \Z_m(a_1, ..., a_n, 0) $.  

Thus  $Z : (\tilde g = 0) \subset \C^{n+1}/\Z_m(a_1, ..., a_n, 0)$.

\smallskip
In the general case, as in \cite{Re87}, Proposition 3.6, we take the $\Z_r$-Galois cover $\pi:Z' \ra Z$,  such that $Z'$ is normal, $\pi$ is etale over $Z \setminus P$, $\pi ^{-1}(P)=: Q$ is a single point and the $\Q$-divisor $\pi^*W := W'$ is a Cartier divisor: $W': (f'=0) \subset Z'$.

The map $W' \ra W$ is an etale cover of $W$ ramified at $P$ and it depends on (a subgroup of) the local fundamental group $\pi_1(W\setminus \{0\})$. By our assumption on the dimensions and Lefschetz theorem this is equal to $\pi_1(\C^{n} / \Z_m(a_1, ..., a_n)\setminus \{0\}) = \Z_m$. 
Therefore the etale cover extends to $\C^{n} / \Z_m(a_1, ..., a_n)$ and we have that $W' :(g'=0) \subset  \C^{n} / \Z_s(a_1, ..., a_n)$, with  $m = r ^. s$. By the first part of the proof  
$Z' : (\tilde g' = 0) \subset \C^{n+1}/\Z_s(a_1, ..., a_n, 0) $.  Therefore  $Z : (\tilde g:= \tilde g' \circ \pi ^{-1} = 0) \subset \C^{n+1}/\Z_m(a_1, ..., a_n, a_{n+1}) $.
\end{proof}

\section{Lifting Weighted Blow-Ups}

This section is dedicated to the proof of Theorem \ref{lifting}; therefore  $f: X \ra Z$ will be a local, projective, divisorial contraction which contracts an irreducible divisor $E$ to $P\in Z$. We assume that $X$  (as a projective variety over $Z$) and $Z$ (as affine variety) are $\Q$-factorial; factoriality on $Z$  depends only on the analytic type of the singularities, on $X$ also on their relative position.

\smallskip
By assumption $Y \subset X$ is a $f$- ample Cartier divisor such that $f' = f_{|Y}: Y \ra f(Y) = W$  is a $\sigma'=  (a_1, \ldots,a_{n-1})$-blow-up, $\pi_{\sigma'}:  Y \ra W$.

In particular $W = (g=0) \subset \C^{n-1} / \Z_m(a_1, ..., a_{n-1})$, possibly with $g \equiv 0$. Proposition \ref{hyperquotient} implies that $Z =  (\tilde g=0) \subset \C^n / \Z_m(a_1, ..., a_{n-1}, a_n)$. Note that $W = f(Y)$ is given as $(x_n = 0) \subset Z$.

\smallskip
We have also  $Pic (Y/W) =\   <L_{|E}>$, where $L= -ME$,  $M=\mathrm{lcm}(a_1,\ldots,a_{n-1})$. 
By the relative Lefschetz theorem, $Pic(X/Z) = Pic (Y/W)=\  <L>$; note that we simply use the injectivity of the restriction map $Pic(X/Z) \lra Pic (Y/W)$, 
true even in the singular case (see for instance p.305 \cite{Kle} or \cite{SGA II}). 

\smallskip
Since $Y$ is Cartier and ample, there exists a positive integer $a$ such that $\cO_X(Y) \sim_f  aL$. We claim that $a_n= aM$. To show this consider the $\sigma := (a_1, ...., a_n)$-blow up of $Z$, 
$\tilde f: \tilde X \ra Z$. Let $\tilde E$ be the exceptional divisor. Note that $Y$ sits in $\tilde X$ as an ample divisor, therefore by Lefschetz theorem there exists 
a Cartier divisor $\tilde L$ on $\tilde X$ which extends $L_{|E'}$, $\tilde L = - M\tilde E$ and $Y = -aM\tilde E$.
Since $\tilde f(\tilde Y): (x_n= 0)$, by Lemma \ref{push-forward} we compute that $a_n= \sigma\textrm{-wt}(x_n) = aM$.

\smallskip
The map $f$ is proper, so, as in Section \ref{s_weighted}, we can apply Grothendieck's language, section 8 of \cite{EGA II}, to say that
$$
X= \Proj_Z (\cO_Z \oplus  \bigoplus_{d > 0}  I_d),
$$ 
where $I_d:=f_*\m O_X(-d (ME)) = f_*\m O_X(dL)$. 

Note that, since $E$ is effective, $I_d=f_*\m O_X(dL) \subset \cO_Z \subset \C^n [x_1, ..., x_n]$ is an ideal for positive $d$ and 
$I_d=f_*\m O_X(dL)=\cO_Z \subset \C^n [x_1, ..., x_n]$ for non positive d.

\smallskip
By Propositions \ref{ideal} and \ref{ideal2}, $X$ will be the weighted blow-up if for positive $d$

$$
f_*\m O_X (dL)= i^{-1} (x_1^{s_1}\cdots x_n^{s_n} : \sum_{j=1}^{n}  s_ja_j \ge  db) ^. \cO_Z
$$
where $b = M$, $s_i$ are non negative integers and $i: Z \ra \C^{n}/ \Z_m(a_1, ...., a_n)$ is the inclusion.

\medskip
We now mimic the proof of Theorem 3.6 in \cite{Mo75}.

Consider the exact sequence
\begin{align}\label{sequence1}
0 \to \m O_X(iL -aL) \to \m O_{X} (iL)\to \m O_{Y}(iL) \to 0, 
\end{align}
for every integer $i$.

We have noticed in Section \ref{s_weighted} that  $R^1{f'}  _*\cO_Y(iL) = 0 $ for $i \in \Z$. Therefore, by  \ref{sequence1},
we obtain surjections 
$R^1{f}  _*\cO_X((i-aj)L) \ra R^1{f}  _*\cO_X(iL) $ , $i, j \in \Z, j \geq 0$.
On the other hand $R^1{f}  _*\cO_X(-jL) = 0 $ for sufficently large $j$. Hence we obtain
$$R^1{f}  _*\cO_X(iL) = 0 \hbox{\ \ \ for every integer\ \ \ } i.$$

\smallskip
All this implies the following exact sequences of $\cO _Z$-algebras, $\cO_Z = \big(\C[x_1, ...,x_n] /(\tilde g)\big)^{\Z_m}$:

\begin{align}\label{sequence2}
0 \to  f_* \m O_X((i-a)L )\rightarrow f_*\m O_{X} (iL)\to f_*\m O_{Y}(iL) \to 0.
\end{align}

In particular, for $i=a$, we have 
$$0 \to \cO_Z  \rightarrow f_*\m O_{X} (aL)\to f_*\m O_{Y}(aL) \to 0.$$

 Let $\theta$ be the image of $1$ by the map $\cO_Z  \rightarrow f_*\m O_{X} (aL)$; then
 \ref{sequence2} becomes
 \begin{align}\label{sequence3} 
 0 \to  f_* \m O_X((i-a)L ) \stackrel{\times \theta}\rightarrow f_*\m O_{X} (iL)\to f_*\m O_{Y}(iL) \to 0;
 \end{align}

$\times \theta$ is exactly $\times (x_n)$.

\medskip
We will prove, by induction on $d$, that 
$$ f_*\m O_X (dL)=(x_1^{s_1}\cdots x_n^{s_n} : \sum_{j=1}^{n}  s_ja_j \ge  db) ^. \cO_Z.$$

\smallskip
By assumption we have that
$$
f_*\m O_{Y}(dL)=(x_1^{s_1}\cdots x_n^{s_n} : \sum_{j=1}^{n-1}  s_ja_j \ge  db)^. \cO_W
$$
where $s_j \in \bb N$.   

\smallskip
By induction on $d$, we can assume that
$$
f_*\m O_{X}((d-a) L)=(x_1^{s_1}\cdots x_n^{s_n} : \sum_{j=1}^{n}  s_ja_j \ge  (d-a)b) ^. \cO_Z,
$$
the case $d - a \leq 0$ being trivial.

\medskip
Let $g=x_1^{s_1}\cdots x_n^{s_n} \in f_*\m O_{X}(dL)$ be a monomial.

\smallskip
If $s_n \ge 1$ then, looking at the sequence \ref{sequence3}, $g$ comes from $f_*\m O_{X}((d-a)L)$ by the multiplication by $(x_{n})$; therefore
$$
\sum_{j=1}^{n}  s_ja_j  = \sum_{j=1}^{n-1}  s_ja_j  +s_n a_n \geq (d-a)b + s_n a_n \geq db -ab + ab = db.
$$

If $s_n=0$, then  $g \in f_*\m O_{Y}(dL)$ and so 
$$
\sum_{j=1}^n  s_ja_j = \sum_{j=1}^{n-1}  s_ja_j  \ge db.
$$
The non-monomial case follows immediately.

\section{Application to MMP with scaling}

The proof of Theorem \ref{MMP}, as explained in the introduction, follows via a standard induction procedure
using Theorem \ref{lifting}, Theorem 1.1 in \cite{AnTa15} and, for dimension $3$, assuming \ref{ass}.
It is actually very similar to the proof of Therem 1.2.A in \cite{AnTa15}, we rewrite it for the reader's convenience.

\begin{proof}[Proof of Theorem \ref{MMP}]

Let $f : X \ra Z$ be a local  projective, divisorial contraction which contracts a prime divisor $E$ to $P\in Z$
as in the Theorem.

$\tau_f(X,L):=\hbox{inf} \{t \in \R : K_X + tL \hbox{ is $f$-nef}\}$ is called the {\it nef-value} of the pair $(f:X \ra Z,L)$. By the rationality theorem of Kawamata (Theorem 3.5 in  \cite{KollarMori}), $\tau_f(X,L):= \tau$ is a rational non-negative number. Moreover $f$ is an adjoint contraction supported by  $K_X+ \tau L$, that is $K_X+ \tau L \lin_f \cO_X$ ($\lin _f$ stays for numerical equivalence over $f$).

By our assumption $\tau > (n-3)$. Therefore $\tau +3 > n > n-1 = dim E$ and, by Proposition 3.3.2 in \cite{AnTa15}, there exists a section of $L$ not vanishing along $E$; in particular $|L|$ is not empty.

\smallskip
Let $H_i \in |L|$ be general divisors for $i=1,\ldots, n-3$.
By Theorem 1.1 in \cite{AnTa15}, quoted in the introduction, for any $i$,  $H_i$ is a variety with terminal singularities and the morphism $f_i=f_{|H_i}: H_i \to f(H_i)=:Z_i$ is a local contraction supported by $K_{H_i}+ (\tau -1 )L_{|H_i}$.
Since $Z$ is terminal and $\Q$-factorial (see \cite[Corollary 3.36]{KollarMori} and \cite[Corollary 3.43]{KollarMori}), then  the $Z_i$'s are  $\Q$-Cartier divisors on $Z$. 
 
For any $t=n-3,\ldots,0$ define $Y_t=\cap_{i=1}^{n-3-t}  H_i$  and $g_t= f_{|Y_t}: Y_t \to f(Y_t)=W_t$; in particular  $Y_{n-3}=X$, $g_{n-3}=f$ and $W_{n-3}=Z$.

By induction on $t$, applying Theorem 1.1 in \cite{AnTa15},  one sees that, for any $t=n-4,\ldots, 0$, $Y_t$ is terminal and $g_t : Y_t \to W_t$ is a local Fano Mori contraction supported by 
$K_{Y_t}+ (\tau -(n - 3 - t )L_{|Y_t}$. Therefore  $W_t$ is a terminal variety (by \cite[Corollary 3.43]{KollarMori}) and it is a $\Q$-Cartier divisor in $W_{t+1}$, because intersection of $\Q$-Cartier divisors (by construction $W_t=\cap_{i=1}^{n-3-t} Z_i$). 
  
Set $L_t:= L_{|W_t}$.  By Proposition 3.3.4 of  \cite{AnTa15} $Bs|L_t|$ has dimension at most 1; by Bertini's theorem (see \cite[Thm. 6.3]{Jou83}) $E_t:= Y_t \cap E$ is a prime divisor.  $E_t$  is the intersection of $\Q$-Cartier divisors and hence it is $\Q$-Cartier.  

Let $X''= Y_0$ and $f''= g_{0}$; by what said above, $f'': X'' \to Z''$ is a divisorial contraction from a 3-fold $X''$ with terminal singularities, which contracts a prime $\Q$-Cartier divisor $E''$ to a point $P\in Z''$.  Using the classification in dimension $3$ of terminal $\Q$-factorial singularities (\cite{Mo82}) and of divisorial contractions (for a summary see \cite{Chen}), one can see that $Z''$ has a hyperquotient singularity at $P$, which is actually contained in a special list. 

By Proposition \ref{hyperquotient} and by induction on $t$, also $Z$ has a hyperquotient singularity at $P$.

\smallskip
Assume now (\ref{ass}), that is that $f''$ is a weighted blow-up of $P$; 
applying Theorem \ref{lifting} inductively on $t$, we have that $f$ is  a weighted blow-up  of a hyperquotient singularities.

\end{proof}

\end{document}